\documentclass[10pt]{article}

\usepackage{amsmath}
\usepackage{amsfonts}
\usepackage{amsthm}
\usepackage{verbatim}
\usepackage{todonotes}

\newtheorem{Th}{Theorem}[section]
\newtheorem{Def}[Th]{Definition}
\newtheorem{Cor}[Th]{Corollary}
\newtheorem{Prop}[Th]{Proposition}
\newtheorem{Lem}[Th]{Lemma}
\newtheorem{Rem}[Th]{Remark}
\newtheorem{Cl}[Th]{Claim}

\DeclareMathOperator{\Vol}{Vol}

\DeclareMathOperator{\inrad}{inrad}

\DeclareMathOperator{\capacity}{cap}

\newcommand{\laplace}{\Delta}
\newcommand{\eigenf}{\varphi_{\lambda}}
\newcommand{\eigenfl}{\varphi_{\lambda}^{loc}}

\newcommand{\NiceNodal}{\Omega_{\varphi_{\lambda}}^*}

\begin{document}
	
\title{On the lower bound of the inner radius of nodal domains}
\author{Bogdan Georgiev}
\date{}

\maketitle

\begin{abstract}
	
	We discuss the asymptotic lower bound on the inner radius of nodal domains that arise from Laplacian eigenfunctions $ \eigenf $ on a closed Riemannian manifold $ (M,g) $.
	
	First, in the real-analytic case we present an improvement of the currently best known bounds, due to Mangoubi (\cite{Man1}). Furthermore, using recent results of Hezari (\cite{Hezari}, \cite{Hezari2}) we obtain $ \log $-type improvements in the case of negative curvature and improved bounds for $ (M,g) $ possessing an ergodic geodesic flow.
	
	Second, we discuss the relation between the distribution of the $ L^2 $ norm of an eigenfunction $ \eigenf $ and the inner radius of the corresponding nodal domains. 
	In the spirit of \cite{Colding-Minicozzi} and \cite{Jacobson-Mangoubi}, we consider a covering of good cubes and show that, if a nodal domain is sufficiently well covered by good cubes, then its inner radius is large.
	
\end{abstract}

\section{Introduction} \label{sec:Intro}

Let $ M $ be a closed Riemannian manifold of dimension $n \geq 3$ with metric $ g $ and denote by $\eigenf$ an eigenfunction of the Laplacian $\laplace$ of $ M $, corresponding to the eigenvalue $\lambda$. Assume that $ \| \eigenf \|_{L^2} = 1 $. We are interested in the geometry of nodal domains in the high-energy limit, i.e. as $ \lambda \rightarrow \infty $. For a readable and far-reaching survey we refer to \cite{Z} and \cite{Z2}.

By a result of Dan Mangoubi (\cite{Man1}), it is known that for a nodal domain $ \Omega_\lambda $, corresponding to $ \eigenf $, the following asymptotic estimate holds:
\begin{equation} \label{eq:Asymptotic-Bounds-Inner-Radius}
	\frac{c_1}{\lambda^{\frac{n-1}{4}+\frac{1}{2n}}} \leq \inrad(\Omega_\lambda) \leq \frac{c_2}{\sqrt{\lambda}},
\end{equation}
where $ c_{1,2} $ depend on $ (M, g) $.

In particular, the asymptotic estimates are sharp in the case of a Riemannian surface, i.e. the inner radius of a nodal domain is comparable to the wavelength $\frac{1}{\sqrt{\lambda}} $. A natural question is whether the mentioned lower bound is optimal also for higher dimensions.

Our first result concerns an improvement in the real-analytic case.

\begin{Th} \label{thm:Inradius-Real-Analytic}
	Let $ (M,g) $ be a real-analytic closed manifold of dimension at least $ 3 $. Let $ \eigenf $ be an eigenfunction of the Laplace operator $ \Delta $ and $ \Omega_\lambda $ be a nodal domain of $ \eigenf $. Then, there exist constants $ c_1 $ and $ c_2 $ which depend only on $ (M,g) $, such that
	\begin{equation}
		\frac{c_1}{\lambda} \leq \inrad (\Omega_\lambda) \leq \frac{c_2}{\sqrt{\lambda}}
	\end{equation}
	Moreover, if $ \eigenf $ is positive (resp. negative) on $ \Omega_\lambda $, then a ball of this radius can be inscribed within a wavelength distance to a point where $ \eigenf $ achieves its maximum (resp. minimum) on $ \Omega_\lambda $. 
\end{Th}

We note that Theorem \ref{thm:Inradius-Real-Analytic} improves Mangoubi's estimates for dimensions $ n \geq 5 $.

The argument consists of two ingredients.

First, we observe that one can almost inscribe a wavelength ball in the nodal domain up to a small in volume error set. In fact, a well-known result due to Lieb (\cite{L}) states that for arbitrary domains $ \Omega $ in $ \mathbb{R}^n $ one can find almost inscribed balls of radius $ \frac{1}{\sqrt{\lambda_1(\Omega)}} $. Furthermore, we refer to \cite{MS} for a result in this spirit stated in terms of capacities.

In \cite{Georgiev-Mukherjee} we were able to obtain a refinement of Lieb's result, specifying the location where a ball of wavelength size can almost be inscribed, as well as the way the error set grows in volume nearby. In particular, wavelength balls can almost be inscribed at points where $ \eigenf $ achieves $ \| \eigenf \|_{L^\infty(\Omega)} $.

Second, one would like to somehow rule out the error set that may enter in the almost inscribed ball near a point of maximum $ x_0 \in \Omega_\lambda $. One way to argue is as follows. Being in the real-analytic setting, eigenfunctions resemble polynomials of degree $ \sqrt{\lambda} $. This observation was utilized in the works of Donnelly-Fefferman (\cite{DF2}) and Jakobson-Mangoubi (\cite{Jacobson-Mangoubi}). What is more, if one takes the unit cube and subdivide it into wavelength-sized small cubes, then these polynomials will be close to their average on most of the small cubes. This implies that the growth of eigenfunctions is controlled on most wavelength-smaller cubes. Now, roughly speaking, we start from a wavelength cube at $ x_0 $ and rescale to the unit cube $ I^n $. Further, $ I^n $ is subdivided into wavelength cubes $ Q_\nu $ and hence most of them will be good. But, if the error set intersects each $ Q_\nu $ deeply it will gain sufficient volume to contradict the volume decay of the first step.
This means that there is a $ Q_\nu $ which is not deeply intersected by the error set.

Moreover, utilizing some recent results of Hezari (\cite{Hezari}) we get that, if one assumes in addition that $ (M,g) $ is negatively curved, then the inradius improves by a factor of $ \log \lambda $. A similar argument works also for $ (M,g) $ with ergodic geodesic flow.

In this note, we also investigate the effect of the moderate growth of $ \eigenf $ on a nodal domain's inner radius. To this end, we exploit a covering by good/bad cubes, inspired by \cite{Colding-Minicozzi} and \cite{Jacobson-Mangoubi}. 	Let us fix a finite atlas $(U_i, \phi_i)$ of $ M $, such that the transition maps are bounded in $C^1$-norm and the metric on each chart domain $U_i$ is comparable to the Euclidean metric in $\mathbb{R}^n$:
\begin{equation} \label{eq:Metric-Comparability}
	\frac{1}{4} e_i \leq g \leq 4 e_i,
\end{equation}
	where we have denoted $ e_i := \phi_i^{*}e $ with $ e $ being the standard Euclidean metric.
	
	We can arrange that $M$ is covered by cubes $K_i \subseteq U_i $, where we decompose $K_i$ into small cubes $K_{ij}$ of size $h$ (to be determined later). Throughout we will denote by $\delta K_{ij}$ the concentric cube, scaled by some fixed scaling factor $ \delta > 1 $.

\begin{Def} \label{def:Good-Bad-Cubes}
	A cube $K_{ij}$ is called $\gamma$-good, if
	\begin{equation}
		\int_{\delta K_{ij}} \eigenf^2 \leq \gamma \int_{K_{ij}} \eigenf^2.
	\end{equation}
	Otherwise, we say that $K_{ij}$ is $\gamma$-bad. We also denote by $\Gamma$ the union of all good cubes (i.e. the good set) and $\Xi := M \backslash \Gamma $.

\end{Def}

We have the following
\begin{Th} \label{th:Main-Theorem}
	
	Let $ (M,g) $ be a smooth closed Riemannian manifold of dimension at least $ 3 $.
	Let $ \Omega_\lambda $ be a fixed nodal domain of the eigenfunction $ \eigenf $.
	
	Then
	\begin{equation}
		\inrad(\Omega_\lambda) \geq C \gamma^{\frac{2-n}{n}} \tau^{\frac{1}{2}} \lambda^{-\frac{1}{2}},
	\end{equation}
	where $ \tau := \int_{\Omega_\lambda \cap \Gamma} \space \eigenf^2 / \int_{\Omega_\lambda} \eigenf^2 $ and  $ C = C(M,g) $.

\end{Th}

Roughly, Theorem \ref{th:Main-Theorem} implies that, if the bulk of the $ L^2 $ norm over the nodal domain is contained in good cubes, then the nodal domain possesses large inner radius.

In Section \ref{sec:Comments} we deduce the following corollaries:

\begin{Cor} \label{cor:Energy-Estimate}
	Let $ (M,g) $ be a smooth closed Riemannian manifold of dimension at least $ 3 $.
	For a nodal domain $ \Omega_\lambda $ of $ \eigenf $, one has
	\begin{equation}
		\inrad(\Omega_\lambda) \geq C \frac{\| \eigenf \|_{L^2(\Omega_\lambda)}^{\frac{2(n-2)}{n}}}{\sqrt{\lambda}},
	\end{equation}
	with $ C = C(M, g)$.
\end{Cor}

Note that the inequality in Corollary \ref{cor:Energy-Estimate} is useful only in dimensions $ 3 $ and $ 4 $, as an application of the standard H\"older inequality gives:
\begin{equation}
	\inrad(\Omega_\lambda) \geq C \frac{\| \eigenf \|_{L^2(\Omega_\lambda)}}{\sqrt{\lambda}},
\end{equation}
which is sharper in higher dimensions.

Moreover, we note that as a by-product we obtain

\begin{Cor} \label{cor:Fat-Nodal-Domain}
	Let $ (M,g) $ be a smooth closed Riemannian manifold of dimension at least $ 3 $.
	There exists a nodal domain of $\eigenf$, denoted by $\NiceNodal$, such that
	\begin{equation*}
		\inrad(\NiceNodal) \asymp \frac{1}{\sqrt{\lambda}},
	\end{equation*}
	In other words, there exist constants $ C_1, C_2 $, depending on $ (M, g) $, such that
	\begin{equation*}
		\frac{C_1}{\sqrt{\lambda}} \leq \inrad(\NiceNodal) \leq \frac{C_2}{\sqrt{\lambda}}
	\end{equation*}
	for $ \lambda $ large enough.
\end{Cor}

As communicated by Dan Mangoubi, Corollary \ref{cor:Fat-Nodal-Domain} also follows directly by looking at a point, where $ \eigenf $ achieves its maximum over $ M $, and further using standard elliptic estimates. Indeed, by rescaling we may assume that $ \eigenf(x_0) = \| \eigenf \|_{L^\infty(M)} = 1 $. Elliptic estimates then imply that $ \| \nabla \eigenf \|_{L^\infty(M)} \leq C \sqrt{\lambda} $ which shows that there is a wavelength inscribed ball at $ x_0 $.

We note that the scaling factor $ \delta > 0 $ also enters in the constants above - however, in our discussion it is fixed and later explicitly set as $ \delta := 16\sqrt{n} $ for technical reasons.

The plan for the rest of this note goes as follows.

In Section \ref{sec:Real-Analytic-Case} we provide the details behind the proof of Theorem \ref{thm:Inradius-Real-Analytic}, following the plan outlined above. We end the Section by discussing the case of $ (M,g) $-negatively curved.

In Section \ref{sec:Asymmetry-History} we recall the essential steps behind the lower bound in (\ref{eq:Asymptotic-Bounds-Inner-Radius}).
Roughly speaking, one cuts a nodal domain $\Omega_\lambda$ into small cubes of size $\inrad(\Omega_\lambda)$. Then, among these, one finds a special cube $ K_{i_0 j_0} $, over which the Rayleigh quotient is carefully estimated. Here an essential role is played by how one compares the volumes of $ \{\eigenf > 0 \}$ and $ \{ \eigenf < 0 \}$ in the special cube (also known as asymmetry estimates).

The motivation behind Theorem \ref{th:Main-Theorem} comes from the question, whether in the special cube $ K_{i_0 j_0} $ one has $ \Vol (\{\eigenf > 0 \}) \sim \Vol (\{\eigenf < 0 \})$. Having this would imply that the inner radius is comparable to the wavelength $\frac{1}{\sqrt{\lambda}} $, i.e. the optimal asymptotic bound.

We also introduce a covering by small good and bad cubes, which arises when one investigates the way the local $L^2$-norm of an eigenfunction $\eigenf$ grows (w.r.t. the domain of integration). Good cubes represent places of controlled $L^2$-norm growth. The motivation behind this consideration is the fact that the volumes of the positivity and negativity set of $\eigenf$ in a good cube are comparable, i.e. their ratio is bounded by constants (\cite{Colding-Minicozzi}).

In Section \ref{sec:Proof-Main} we show the statement of Theorem \ref{th:Main-Theorem} and its Corollaries $ \ref{cor:Fat-Nodal-Domain} $ and $ \ref{cor:Energy-Estimate} $.

We end the discussion by making some further comments in Section \ref{sec:Comments}.

\subsection{Acknowledgements}

I am grateful to Henrik Matthiesen, Mayukh Mukherjee, Alexander Logunov, Dan Mangoubi, Hamid Hezari and Steve Zelditch for their encouragement, comments and valuable remarks. I would also like to thank my supervisor Werner Ballmann for the constant support and all the knowledge.

\section{The inradius of nodal domains in the real analytic case} \label{sec:Real-Analytic-Case}

We consider the case of a real analytic manifold $ (M,g) $ of dimension at least $ 3 $.
A main insight in this situation is that polynomials approximate eigenfunctions sufficiently well, i.e. an eigenfunction $ \eigenf $ exhibits a behaviour, which is similar to that of a polynomial of degree $ \sqrt{\lambda} $. We refer to the papers of Donnelly-Fefferman \cite{DF2} and Jakobson-Mangoubi \cite{Jacobson-Mangoubi} for a vivid illustration of this observation.

We now prove Theorem \ref{thm:Inradius-Real-Analytic}.

\begin{proof}[Proof of Theorem \ref{thm:Inradius-Real-Analytic}]
	
	Let us assume without loss of generality that $ \eigenf $ is positive on $ \Omega_\lambda $ and let $ x_0 \in \Omega_\lambda $ be a point where $ \eigenf $ reaches a maximum on $ \Omega_\lambda $. 
	First, an examination of the proof of Theorem $ 1.3 $ of \cite{Georgiev-Mukherjee}, gives
	\begin{equation} \label{eq:Large-Volume-near-Max-Point}
		\frac{\Vol(B_{\frac{r}{\sqrt{\lambda}}} (x_0) \cap \Omega_\lambda ) }{\Vol (B_{\frac{r}{\sqrt{\lambda}}} (x_0))} > 1 - c r^{\frac{2n}{n-2}},
	\end{equation}
	
	where $ r $ is sufficiently small. This quantitatively means that the more we shrink the ball $ B_{\frac{r}{\sqrt{\lambda}}} (x_0) $, the "more inscribed" it becomes having a small error set, outside of the nodal domain $ \Omega_\lambda $. We remark that the existence of such an almost inscribed wavelength ball could also be deduced from the corresponding statement in $ \mathbb{R}^n $  (cf. \cite{L}) and a partition of unity argument. However, the result in \cite{Georgiev-Mukherjee} specifies also the location of such a ball, i.e. at a point $ x_0 $ where $ \phi_\lambda $ achieves a maximum.
	
	We would like to take $r$ sufficiently small (to be determined later), and then rescale the ball $B_{\frac{r}{\sqrt{\lambda}}}$ to the unit ball $B_1$ in  $\mathbb{R}^n$. We will denote the rescaled eigenfunction on the unit ball by $\eigenfl$.
	
	We now recall and adapt results from \cite{Jacobson-Mangoubi}. First, we recall the following (cf. (Proposition 4.1,\cite{Jacobson-Mangoubi}))
	\begin{Lem} \label{lem:Preferred-Cube}
		Let $ (M,g) $ be real-analytic.
		There exists a cube $Q \subseteq B_1$, which does not depend on $\eigenfl$ and $\lambda$, and has the following property: suppose $\delta > 0$ is taken, so that $\delta \leq \frac{C_1}{\sqrt{\lambda}}$. We decompose $Q$ into smaller cubes $Q_\nu$ with sides in the interval $(\delta, 2\delta)$. Then, there exists a subset of $E_\epsilon \subseteq Q$ of measure $| E_\epsilon | \leq C_2 \epsilon \sqrt{\lambda} \delta $, so that
		\begin{equation}
			\frac{1}{C_3(\epsilon)} \leq \frac{(\eigenfl(x))^2}{Av_{(Q_\nu)_x} (\eigenfl)^2} \leq C_3(\epsilon), \quad \forall x \in Q \backslash E_\epsilon,
		\end{equation}
		with $C_3(\epsilon) \rightarrow \infty $ as $\epsilon \rightarrow 0$. The constants $C_1, C_2, C_3$ do not depend on $\eigenf$ and $\lambda$. The notation  $ Av_{{(Q_\nu)}_x} F$ denotes the average of $ F $ over a cube $ Q_\nu $ which contains $ x $.
	\end{Lem}
	Lemma \ref{lem:Preferred-Cube} stems from the corresponding result for holomorphic functions $ F $ (cf. Propositions 3.2 and 3.7, \cite{Jacobson-Mangoubi}), defined on the unit cube $ I^n $ and satisfying the following growth condition:
	\begin{equation} \label{eq:Holomorphic-Growth}
		\sup_{B_2^n} |F| \leq |F(0)|e^{C\sqrt{\lambda}},
	\end{equation}
	where $ B_2^n = B_2 \times \dots \times B_2 \subseteq \mathbb{C}^n $ denotes the $ 2 $-polydisk. 
	To transfer the statement to the local eigenfunction $ \eigenfl $ one observes that $ \eigenfl $ admits a holomorphic continuation on a ball, whose size does not depend on $ \lambda $ (Lemma 7.1, \cite{DF2}). Here one uses the fact that on a wavelength scale $ \eigenfl $ is almost harmonic, i.e. it is a solution to slight perturbation of the standard Laplace equation. In other words, $ \eigenfl $ is holomorphically extended to a ball $ \| z \| \leq \rho $, where $ \rho $ does not depend on $ \lambda, \eigenfl $. Moreover,
	\begin{equation} \label{eq:holo-continuation}
		\sup_{\| z \| \leq \rho} |\eigenfl(z)| \leq C \sup_{B_1} |\eigenfl(x)|, 
	\end{equation}
	where $ C $ again does not depend on $ \lambda, \eigenfl $. In this direction we also remark that, in general, $ \eigenf $ extends analytically to a uniform Grauert tube of $ M $, having radius which is independent of $ \lambda $ - we refer to Sections $ 1 $ and $ 11 $, $ \cite{Z3} $.
	
	We now subdivide the cube $ Q $ into small cubes $ Q_\nu $ for which Lemma \ref{lem:Preferred-Cube} holds.
	\begin{Def}
		$ Q_\nu $ is called $ E_\epsilon $-\textit{good}, if
		\begin{equation}
			\frac{| E_\epsilon \cap Q_\nu |}{|Q_\nu|} < 10^{-2n}\omega_n,
		\end{equation}
		where $ \omega_n $ denotes the volume of the unit ball in $ \mathbb{R}^n $. Otherwise, $ Q_\nu $ is $ E_\epsilon $-\textit{bad}.
	\end{Def}
	
	It turns out that the $ E_\epsilon $-good cubes $ Q_\nu $ are not that different from the good cubes, defined in Section \ref{sec:Asymmetry-History} (cf. also Lemma 5.3, \cite{Jacobson-Mangoubi}). We have
	
	\begin{Lem}\label{lem:Good-Ball}
		Let $ Q_\nu $ be an $ E_\epsilon $-good cube.
		Let $ B \subseteq 2B \subseteq Q_\nu $ be a ball centered somewhere in $ \frac{1}{2} Q_\nu $, whose size is comparable to the size of $ Q_\nu $. Then
		\begin{equation}
			\frac{\int_{2B} (\eigenfl)^2}{\int_B (\eigenfl)^2} \leq \tilde{C_1} C_3(\epsilon),
		\end{equation}
		where $ C_3(\epsilon) $ comes from Lemma \ref{lem:Preferred-Cube} and $ \tilde{C_1} $ depends only on the dimension $ n $.
	\end{Lem}
	
	\begin{Lem}
		The proportion of bad cubes to all cubes is smaller than $ \tilde{C_2} |E_\epsilon| $, where $ \tilde{C_2} $ depends only on the dimension.
	\end{Lem}

	By fixing $ \epsilon $ sufficiently small, the previous two Lemmata imply that we can arrange that $ 90\% $ of the small cubes $ Q_\nu $ to be $ E_\epsilon $-good. Moreover, each $ E_\epsilon $-good cube possesses a comparable inscribed ball, having a fixed growth exponent $ C_\epsilon $.

	Now, let us define the error set (or "spike") $ S := B_1 \backslash \Omega_\lambda $. The volume decay (\ref{eq:Large-Volume-near-Max-Point}) gives
	\begin{equation} \label{eq:Spike-Estimate}
		\frac{|S|}{|B_1|} \leq c r^{\frac{2n}{n-2}}.
	\end{equation}
	Now, suppose that $ S $ intersects each small $ E_\epsilon $-good cube $\frac{1}{2} Q_\nu $. Otherwise, there will be an inscribed cube of radius $ \frac{C}{{\lambda}} $ in the nodal domain $ \Omega_\lambda $ and the claim follows. 
	
	There are two cases:
	\begin{enumerate}
		\item Suppose that in a $ E_\epsilon $-good cube $ Q_\nu $ the nodal set does not intersect $ \frac{1}{2}Q_\nu $. This means that $ \frac{1}{2} Q_\nu \subseteq S $, hence
		\begin{equation}
			\frac{|S \cap Q_\nu|}{|Q_\nu|} \geq \frac{1}{2^n}.
		\end{equation}
		
		\item Suppose that the nodal set intersects $ \frac{1}{2}Q_\nu $. Since $ Q_\nu $ is $ E_\epsilon $-good, we can find a ball $ \tilde{B} $ as in Lemma \ref{lem:Good-Ball}, so that its center lies on the nodal set. Elliptic estimates (cf. Proposition \ref{prop:Controlled-Asymmetry} below and Proposition 5.4, \cite{Jacobson-Mangoubi}) imply that
		\begin{equation}
			\frac{|\{ \eigenfl < 0 \} \cap \tilde{B}|}{|\tilde{B}|} \geq C.
		\end{equation}
		By definition $ \{ \eigenfl < 0 \} \cap \tilde{B} \subseteq S \cap \tilde{B} $, so we get
		\begin{equation}
			\frac{|S \cap Q_\nu|}{|Q_\nu|} \geq C.	
		\end{equation}
		
	\end{enumerate}
	
	Summing up the two cases over all $ E_\epsilon $-good cubes, we see that
	\begin{equation}
		\frac{| S\cap Q|}{|Q|} \geq C.
	\end{equation}

	By using the estimate (\ref{eq:Spike-Estimate}) and selecting $ r $ sufficiently small, we arrive at a contradiction. This means that there is a $ E_\epsilon $-good cube $ Q_\nu $, so that $ \frac{1}{2} Q_\nu \subseteq \Omega_\lambda $
	
\end{proof}

\begin{Rem}
	We note that the statement extends also to points $ x_0 $ at which the eigenfunction almost reaches its maximum on $ \Omega_\lambda $ in the sense, that
	\begin{equation}
		C \eigenf(x_0) \geq \| \eigenf \|_{L^\infty(\Omega_\lambda)},
	\end{equation}
	for some fixed constant $ C > 0 $. In particular, if there are multiple "almost-maximum" points $ x_0 $, there should be an inscribed ball of radius $ \frac{1}{\lambda} $ near each of them.
\end{Rem}

\begin{Rem}
	Let us observe that the estimates essentially depend on the growth of $ \eigenf $ at $ x_0 $. We have used the upper bound $ C\sqrt{\lambda} $ on the doubling exponent in the worst possible scenario as shown by Donnelly-Fefferman. It is believed that $ \eigenf $ rarely exhibits such an extremal growth. If the growth is better, this allows to take larger  cubes $ Q_\nu $ and the bound on the inner radius improves. In particular, a constant growth implies the existence of a wavelength inscribed ball.
\end{Rem}

To conclude this section, let us briefly mention that the case of $ (M, g) $ being a closed smooth Riemannian manifold could be treated by a combinatorial approach. We address these issues in a further note.

\subsection{The quantum ergodic case}

First, we mention some recent results of H. Hezari (\cite{Hezari} and \cite{Hezari2}), addressing quantum ergodic sequences of eigenfunctions. Let us assume that $ (M,g) $ is a closed Riemannian manifold with negative sectional curvature. Let $ ({{\eigenf}_i}) $ be any orthonormal basis of $ L^2(M) $, where $ ({{\eigenf}_i}) $ are eigenfunctions with eigenvalues $ \lambda_i $. Then, for a given $ \epsilon > 0 $, there exists a density-one subsequence $ S_\epsilon $, so that
\begin{equation} \label{eq:Hezari-Improvement}
	a_1 (\log \lambda_j)^{\frac{(n-1)(n-2)}{4n^2} - \epsilon} \lambda_j^{-\frac{1}{2} - \frac{(n-1)(n-2)}{4n}} \leq \inrad(\Omega_\lambda)
\end{equation}

We refer to \cite{Hezari} for further details.

The heart of Hezari's argument lies in observing that growth exponents (i.e. doubling exponents) improve, provided that eigenfunctions equidistribute at small scales (cf. \cite{Hezari2}). More precisely, if we assume that for some small $ r > \frac{1}{\sqrt{\lambda}} $, we have
\begin{equation}
	K_1 r^n \leq \int_{B_r(x)} |\eigenf|^2 \leq K_2 r^n,
\end{equation}
for $ K_1, K_2 $ fixed constants and all geodesic balls $ B_r(x) $,
then
\begin{equation}
	\log \left( \frac{\sup_{B_{2s}(x)} |\eigenf|^2}{\sup_{B_s(x)} |\eigenf|^2} \right) \leq C r \sqrt{\lambda}.
\end{equation}
Here the statement holds for all $ s $ smaller than $ 10r $. In particular, in the negatively curved setting, results of $ \cite{Hezari-Riviere} $ give that $ r $ above could be taken as $ (\log \lambda)^{-k} $ for any $ k \in (0, \frac{1}{2n}) $. Further, plugging the improved growth bound into Mangoubi's argument (outlined in Section \ref{sec:Asymmetry-History}) gives the improved estimate (\ref{eq:Hezari-Improvement}).

Now, let us assume that $ (M,g) $ is real-analytic and negatively curved. Then, Hezari's improved growth bound affects the estimate (\ref{eq:Holomorphic-Growth}) and Lemma \ref{lem:Preferred-Cube}, showing that we may subdivide the preferred cube into smaller cubes $ Q_\nu $ of size $ \frac{(\log \lambda)^k}{\sqrt{\lambda}} $, most of which will be good as prescribed by the Lemma. Continuing the argument in the proof of Theorem \ref{thm:Inradius-Real-Analytic}, we arrive at
\begin{equation}
	\inrad (\Omega_\lambda) \geq \frac{C (\log \lambda)^k}{\lambda},
\end{equation}
where $ k $ could be taken as any number in $ (0, \frac{1}{2n}) $.

\section{Inradius estimates via asymmetry} \label{sec:Asymmetry-History}

From now on we will be considering a closed smooth Riemannian manifold $ (M,g) $.
Let us consider an eigenfunction $ \eigenf $ and an associated nodal domain $\Omega_\lambda \subseteq M$.

First, we recall the main ideas behind the asymptotic lower bound:

\begin{equation}
	\frac{c_1}{\lambda^{\frac{n-1}{4}+\frac{1}{2n}}} \leq \inrad(\Omega_\lambda),
\end{equation}
as outlined in $ \cite{Man1} $.

We denote by $\psi$ the restriction of $ \eigenf $ to $ \Omega_\lambda $, extended by $ 0 $ to $ M $. Then $ \psi $ realizes the first Dirichlet eigenvalue of $\Omega_\lambda$, i.e.
\begin{equation*}
	\frac{\int_{\Omega_\lambda} |\nabla \psi|^2 d\Vol}{\int_{\Omega_\lambda} | \psi|^2 d\Vol} = \lambda_1(\Omega_\lambda) = \lambda.
\end{equation*}

	We may assume that $\phi_i(U_i)$ is a cube $K_i$, whose edges are parallel to the coordinate axes, and we can further cut it into small non-overlapping small cubes $K_{ij} \subseteq K_i$ of appropriately selected side length $h$, comparable to $\inrad(\Omega_\lambda)$.
	
	Having this construction in mind, we define \textit{local Rayleigh quotients}:

\begin{Def}
	A local Rayleigh, associated to the eigenfunction $\psi$ and decomposition $\{ K_{ij} \}$ as above, is the quantity
	\begin{equation}
		R_{ij}(\psi) := \frac{\int_{\phi_i^{-1}(K_{ij})} | \nabla \psi|^2 d\Vol }{\int_{\phi_i^{-1}(K_{ij})} | \psi|^2 d\Vol}.
	\end{equation}
\end{Def}

	The main idea is to find a specific small cube $K_{i_0 j_0}$, so that $R_{i_0 j_0}$ is bounded in a suitable way from above and from below. To be more precise, one has the following:
	
\begin{Prop} \label{prop:Reyleigh-Bounds}
	Having the decomposition $\{ K_{ij} \}$ and eigenfunction $\psi$ as above, there exists a cube $K_{i_0 j_0}$, such that:

\begin{equation} \label{eq:Medium-Rayleigh}
	C \frac{\left( \lambda^{\frac{1-n}{2}} \right)^{1-\frac{2}{n}}}{h^2} \leq R_{i_0 j_0} \leq m \lambda_1(\Omega_\lambda) = m \lambda,
\end{equation}

where $C = C(M,g)$ is a constant and $m$ denotes the number of the charts $\{( U_i, \phi_i) \}$.
\end{Prop}

Note that (\ref{eq:Medium-Rayleigh}) leads to the asymptotic lower bound on $ \inrad(\Omega_\lambda) $, because of the choice $ h \sim \inrad(\Omega_\lambda) $.

\begin{proof} [Proof of Proposition 
\ref{prop:Reyleigh-Bounds}]
	First, there exists a cube $K_{i_0 j_0}$, for which the upper bound in (\ref{eq:Medium-Rayleigh}) holds, since otherwise one has a contradiction with the definition of $ \psi $ by summing over all small cubes.
	
	Note that we may assume w.l.o.g. that $ \psi $ is negative on $ \Omega_\lambda $ and then, by definition, one gets
	\begin{equation} \label{eq:Psi-Eigenf}
		\Vol(\{ \psi = 0 \}\cap K_{i_0 j_0}) \geq \Vol(\{ \eigenf>0 \}\cap K_{i_0 j_0}),
	\end{equation}
	
	One now proceeds to bound $ R_{i_0 j_0} $ from below, using consecutively a Poincare inequality and a capacity estimate (Sections 10.1.2 and 2.2.3, respectively, from \cite{Maz}):

	\begin{gather*} \label{eq:Poincare-Inequalities}
		R_{i_0 j_0} \geq C \frac{\capacity_2(\{ \psi = 0 \} \cap K_{i_0 j_0}, 2 K_{i_0 j_0})}{h^n} \geq \tilde{C} \frac{\left( \Vol(\{ \psi = 0 \}\cap K_{i_0 j_0}) \right)^\frac{n-2}{n}}{h^n} \geq \\
		\geq \tilde{C} \frac{\left( \Vol(\{ \eigenf>0 \}\cap K_{i_0 j_0}) \right)^\frac{n-2}{n}}{h^n}.
	\end{gather*}
	with $ C, \tilde{C} $ depending on $ M $.

	Thus, it is natural to search for a lower bound on the volume of the positivity set of $ \eigenf $ contained in $ K_{i_0 j_0} $.
	
	By using maximum principles for elliptic PDE in combination with a growth bound, due to Donnelly-Fefferman (\cite{DF2}), D. Mangoubi showed (\cite{Man1}) that the following asymmetry estimate holds:
	\begin{equation}
		\frac{\Vol(\{ \eigenf>0 \}\cap K_{i_0 j_0})}{\Vol( K_{i_0 j_0})} \geq \frac{C}{\lambda^{\frac{n-1}{2}}}.
	\end{equation}
	Noting that $  \Vol( K_{i_0 j_0})  \sim h^n$ and putting the above chain of inequalities together yields the claim.
\end{proof}

A few comments are in order. First, note that an improvement of the asymmetry of nodal domains leads to a direct improvement of the lower bound for the inner radius. Moreover, it has been suggested (e.g. \cite{NPS}) that the asymmetry of nodal domains may be better than $\frac{C}{\lambda^{(n-1)/2}}$.

Second, in some sense it has been shown (e.g. \cite{DF2}; Lemma $ 4 $ in \cite{Colding-Minicozzi}), that there are many small cubes where the asymmetry of a nodal domain is already far better (in fact asymptotically optimal). In other words, one should expect that on a lot of small cubes $ \Vol (\{\eigenf > 0 \}) \sim \Vol (\{\eigenf < 0 \})$.

These remarks suggest that, in order to improve the lower inradius bound, it would be desirable to exhibit a better asymmetry of $\Omega_\lambda$ in the specific cube $K_{i_0 j_0}$.

\subsection{Good cubes and bad cubes} \label{sec:Cubes}

Let us fix an eigenfunction $\eigenf$ of $\laplace$ with eigenvalue $\lambda$ and $\| \eigenf \|_{L^2(M)} = 1$.

As above we consider the finite atlas $\{ (U_i, \phi_i) \}$ of $ M $, and arrange that $M$ is covered by cubes $K_i \subseteq U_i $, where $K_i$ is decomposed into small cubes $K_{ij}$ of size $h$ (to be determined later). Again denoting by $\delta K_{ij}$ the concentric cube, scaled by some fixed scaling factor $ \delta > 1 $, we may also assume that $\delta K_{ij} \subseteq U_i$.

We note that the metric $g$ is comparable to the Euclidean one on each cube $K_i$ and, moreover, each point $x \in M$ is contained in at most $\kappa_\delta$ of the concentric cubes $\delta K_{ij}$, where $\kappa_\delta$ is some constant, not depending on the chosen cube size $h$.

In the light of Definition \ref{def:Good-Bad-Cubes}, we have that the covering $ K_{ij} $ is divided into good and bad cubes.

First, we note that the covering is robust, in the sense that the good cubes can be arranged to consume most of the $ L^2 $ norm. Essentially, by using the definition one can show:

\begin{Lem} \label{lem:L2-norm-lower-bound}
	We have
	\begin{equation}
		\int_{\Gamma} \eigenf^2 \geq 1 - \frac{\kappa_\delta}{\gamma}.
	\end{equation}
\end{Lem}

\begin{proof}
	Using $\| \eigenf \|_{L^2} = 1$, we have
\begin{multline*}
		\int_{\Gamma} \eigenf^2 \geq 1 - \int_{\Xi} \eigenf^2 \geq 1 - \sum_{K_{ij}\text{-bad}}	\int_{K_{ij}} \eigenf^2 \geq 1 - \sum_{K_{ij}\text{-bad}} \frac{1}{\gamma} \int_{\delta K_{ij}} \eigenf^2 \geq \\ \geq 1 - \frac{\kappa_\delta}{\gamma} \int_M \eigenf^2 = 1 - \frac{\kappa_\delta}{\gamma}.\\
\end{multline*}
\end{proof}

Again, we note that without any dependence on $ \lambda $ or the size of the small cubes $h$ one is able to control how big (in $ L^2 $ sense) the good set $ \Gamma $ is.

\section{Proof of Theorem \ref{th:Main-Theorem}}\label{sec:Proof-Main}

We now show how the portion of the $ L^2 $ norm over a nodal domain, occupied by good cubes, gives a lower bound on the inner radius.
Roughly, having a lot of good cubes over a nodal domain increases the chance that $ K_{i_0 j_0} $ is a good one.
As in the method of Mangoubi, sketched in Section \ref{sec:Asymmetry-History}, we find a special small cube and estimate the corresponding Rayleigh quotient in a similar way.
However, we will have the advantage that the special cube is also good, which would lead to optimal asymmetry.

\begin{proof} [Proof of Theorem \ref{th:Main-Theorem}]

\begin{Cl} \label{cl:The-Good-Cube}
	There exists a good cube $K_{i_0 j_0}$, such that
	\begin{equation} \label{eq:The-Good-Cube-Rayleigh}
		R_{\left(\delta K_{i_0 j_0} \right)} (\psi) \leq \frac{{\kappa_\delta}}{\tau} \lambda_1(\Omega_\lambda),
	\end{equation}
	where $ \psi $ is defined as in Section \ref{sec:Asymmetry-History} and $R_{\left(\delta K_{i_0 j_0} \right)} (\psi)$ denotes the local Rayleigh quotient w.r.t. the cube $\delta K_{i_0 j_0}$. As above, $ \kappa_\delta $ denotes the maximal number of cubes $ \delta K_{ij} $ that can intersect at a given point.
\end{Cl}

\begin{proof}
	First, let us denote by $\delta \Gamma$ the union of all good cubes scaled by a factor of $\delta > 1$.
	Assuming the contrary, we get:
	\begin{multline}
	\int_{\Omega_\lambda} |\nabla \psi|^2 \geq \int_{\Omega_\lambda \cap \delta\Gamma} |\nabla \psi|^2 \geq \frac{1}{\kappa_\delta} \sum_{K_{ij}\text{-good}} \int_{\Omega_\lambda \cap \delta K_{ij}} |\nabla \psi|^2  > \\
	> \frac{1}{\tau}\lambda_1(\Omega_\lambda) \sum_{K_{ij}\text{-good}} \int_{\Omega_\lambda \cap \delta K_{ij}} |\psi|^2 \geq \frac{1}{\tau}\lambda_1(\Omega_\lambda)\int_{\Gamma} | \psi|^2 \geq \\ 
	\geq \lambda_1(\Omega_\lambda) \int_{\Omega_\lambda} |\psi|^2.\\
	\end{multline}
	Hence, a contradiction with the definition of $\psi$.
\end{proof}

This means that in the cube $\delta K_{i_0 j_0}$ we have the upper bound from (\ref{eq:Medium-Rayleigh}). However, we have the advantage that $K_{i_0 j_0}$ is $\gamma$-good - this implies that the asymmetry and the geometry of the nodal set is under control.

From now on let us fix $ \delta := 16\sqrt{n} $.
The following proposition is similar to Proposition 1 in \cite{Colding-Minicozzi} and Proposition 5.4 in \cite{Jacobson-Mangoubi}. We supply the technical details for completeness:

\begin{Prop} \label{prop:Controlled-Asymmetry}
	Let $\gamma, \rho > 1$ be given. Then there exists $\Lambda > 0$, such that, if one takes the cube size $r \leq \frac{\rho}{\sqrt{\lambda}}$ for $\lambda \geq \Lambda$ and assumes that $\eigenf$ vanishes somewhere in $\frac{1}{2} K_{i_0 j_0}$, then
	\begin{equation}
		\frac{\Vol(\{\eigenf > 0\} \cap \delta K_{i_0 j_0})}{\Vol(\delta K_{i_0 j_0})} \geq \frac{C}{\gamma^2},
	\end{equation}
	where $C$ depends on $ n, \rho, (M, g)$.
\end{Prop}

The same holds for the negativity set. Hence the asymmetry of $\Omega_\lambda$ in $\delta K_{i_0 j_0}$ is bounded below by the constant $C/\gamma^2 > 0$, which essentially depends on the good/bad growth condition and not on $\lambda$.

\begin{proof} (of Proposition)
	We denote by $ K_r(p) $ the cube of edge size $ r $ centered at $ p $, whose edges a parallel to the coordinate axes. We also denote by $ B_r(p) $ a metric ball (w.r.t the metric $ g $) of radius $ r $ centered at $ p $. Let us assume that $  K_r(p) := K_{i_0 j_0} $. By the metric comparability (\ref{eq:Metric-Comparability}), we have:
	\begin{equation}
		B_{\frac{r}{4}} \subseteq K_r(p) \subseteq B_{\sqrt{n}r}(p).
	\end{equation}
	
	Recall the following generalization of the mean value principle (Lemma 5, \cite{Colding-Minicozzi}):
	\begin{Lem} \label{lem:Generalized-Mean-Value}
		There exists $ R = R(M, g) > 0 $, such that if $ r \leq R $ and $ \eigenf(p) = 0 $, then
		\begin{equation}
		\left| \int_{B_r(p)} \eigenf \right| \leq \frac{1}{3} \int_{B_r(p)} |\eigenf|.
		\end{equation}
	\end{Lem}
	By assumption, there exists a point $ q \in \frac{1}{2} K_{i_0 j_0}, \eigenf (q) =0 $, so the lemma, in combination with metric comparability, implies that
	\begin{equation}
		\int_{K_r(q)} |\eigenf| \leq \int_{B_{\sqrt{n}r}(q)} |\eigenf| \leq 3 \int_{B_{\sqrt{n}r}(q)} \eigenf^+ \leq 3 \int_{K_{4\sqrt{n}r}(q)} \eigenf^+,
	\end{equation}
	where $ \eigenf^+, \eigenf^- $ respectively denote the positive and negative part of $ \eigenf $.

	Hence,
	\begin{multline}\label{eq:Inequality-Chain1}
		\frac{1}{9}\left( \int_{K_{2r}(q) } |\eigenf| \right)^2 \leq \left( \int_{K_{8\sqrt{n}r}(q)} \eigenf^+ \right)^2 \leq \\
\leq \Vol(K_{8\sqrt{n}r}(q) \cap \{ \eigenf > 0  \}) \int_{K_{16\sqrt{n}r}(q)} \eigenf^2, \\
	\end{multline}
	where we have used the Cauchy-Schwartz inequality.
	
	We estimate further the integral from the last expression:
	\begin{multline}\label{eq:Inequality-Chain2}
	\left( \int_{K_{16\sqrt{n}r}(q)} \eigenf^2 \right)^2 \leq
		\gamma^2 \left( \int_{K_{r}(p)} \eigenf^2 \right)^2 \leq \\ \leq \gamma^2 \left( \int_{K_{2r}(q)} |\eigenf| |\eigenf| \right)^2 \leq \gamma^2 \sup_{K_{2r}(q)} \eigenf^2 \left( \int_{K_{2r}(q)} |\eigenf| \right)^2.
	\end{multline}
	
	Note that, since $ r $ is comparable to wavelength, elliptic estimates (Theorem 1.2, \cite{Li-Schoen}) imply:
	\begin{equation}\label{eq:Inequality-Chain3}
		\sup_{K_{2r}(q)} \eigenf^2 \leq \sup_{B_{2\sqrt{n}r}(q)} \eigenf^2 \leq C_0 r^{-n} \int_{B_{4\sqrt{n}r}(q)} \eigenf^2 \leq C_0 r^{-n} \int_{K_{16\sqrt{n}r}(q)} \eigenf^2,
	\end{equation}
	where $ C_0 = C_0(M, g, \rho, n) $.

	Plugging (\ref{eq:Inequality-Chain3}) into (\ref{eq:Inequality-Chain2}), one gets
	\begin{equation}
		\int_{K_{16\sqrt{n}r}(q)} \eigenf^2 \leq C_0 \gamma^2 r^{-n} \left( \int_{K_{2r}(q)} |\eigenf| \right)^2
	\end{equation}
	and in combination with (\ref{eq:Inequality-Chain1}) this yields
	\begin{equation}
		\frac{r^n}{9C_0 \gamma^2} \leq \Vol(\{ \eigenf > 0  \} \cap K_{8\sqrt{n}r}(q) ) \leq \Vol(\{\eigenf > 0\} \cap \delta K_{i_0 j_0}).
	\end{equation}
	The statement of the proposition follows after dividing by $ \Vol( \delta K_{i_0 j_0})  \leq C_1 r^n$.
	\end{proof}
	
\begin{Rem}
	One may also exhibit a version of Lemma \ref{lem:Generalized-Mean-Value} for cubes, thus making some of the constants better. However, using balls and comparability as above suffices for our purposes.
\end{Rem}

To finish the proof of the main statement, we put together the latter observations.

Again, we consider an atlas and cube decomposition as above. Following \cite{Man1}, we fix the size of the small cube-grid 
\begin{equation}
	h := 8 \max_i r_i,
\end{equation}
where $r_i$ denotes the inner radius of $\Omega_\lambda$ in the chart $(U_i, \phi_i)$ with respect to the Euclidean metric.
	
We consider the cube $ K_{i_0 j_0} $, prescribed by Claim \ref{cl:The-Good-Cube}. The choice of $ h $ ensures that $ \eigenf (q) = 0 $ for some $ q \in \frac{1}{2} K_{i_0 j_0} $. Then the conditions of Proposition \ref{prop:Controlled-Asymmetry} are satisfied. This means that
\begin{equation}
	\Vol(\{\psi = 0\} \cap \delta K_{i_0 j_0}) \geq \Vol(\{\eigenf > 0\} \cap \delta K_{i_0 j_0}) \geq \frac{C}{\gamma^2} h^n.
\end{equation}

We plug the latter in the Poincare and capacity estimates (\ref{eq:Poincare-Inequalities}) and recall that (\ref{eq:The-Good-Cube-Rayleigh}) holds.
We get
\begin{equation}
	C \frac{1}{h^ 2} \left(\frac{1}{\gamma^2}\right)^{\frac{n-2}{n}}\leq \frac{{\kappa_\delta}}{\tau} \lambda_1(\Omega_\lambda).
\end{equation}
A rearrangement gives
\begin{equation}
	h \geq C \left[ \sqrt{\frac{\tau}{\kappa_\delta}} \left( \frac{1}{\gamma}\right)^{\frac{n-2}{n}} \right] \frac{1}{\sqrt{\lambda}}.
\end{equation}
The proof finishes by recalling that $ h \leq 8\inrad(\Omega_\lambda) $ by assumption.
\end{proof}

\section{Some further comments and corollaries} \label{sec:Comments}

Let us briefly explain the Corollaries \ref{cor:Fat-Nodal-Domain} and \ref{cor:Energy-Estimate}.

\begin{proof} [Proof of Corollary \ref{cor:Fat-Nodal-Domain}]

	Let us fix $ \gamma := 4 \kappa_\delta $. A simple summation argument, yields
	
	\begin{Cl} \label{cl:Nice-Nodal-Domain}
		There exists a nodal domain $\NiceNodal$, such that
		\begin{equation}
			 \int_{\Gamma \cap \NiceNodal} \eigenf^2 \geq 3 \int_{\Xi \cap \NiceNodal} \eigenf^2.
		\end{equation}
		In particular,
		\begin{equation}
			 \int_{\Gamma} (\psi^*)^2 \geq 3 \int_{\Xi} (\psi^*)^2,
		\end{equation}
		where $\psi^*$ is the function, which realizes $\lambda_1(\NiceNodal)$, extended by zero outside $\NiceNodal$.
	\end{Cl}
	
	Indeed, assuming the contrary and summing over all nodal domains, one gets a contradiction with Lemma \ref{lem:L2-norm-lower-bound} and the fact that $\| \eigenf \|_{L^2} = 1$.
	
	Now, apply Theorem \ref{th:Main-Theorem} with $ \NiceNodal $.

\end{proof}

We now prove the energy inequality. The idea is just to tailor $ \gamma $ along $ \Omega_\lambda $.

\begin{proof} [Proof of Corollary \ref{cor:Energy-Estimate}]

	In the light of Lemma \ref{lem:L2-norm-lower-bound}, we just take
	\begin{equation}
		\gamma := \frac{4 \kappa_\delta}{\| \eigenf \|_{L^2(\Omega_\lambda)}^2},
	\end{equation}
	thus having
	\begin{equation}
		\int_{\Gamma} \eigenf^2 \geq 1 - \frac{\| \eigenf \|_{L^2(\Omega_\lambda)}^2}{4}.
	\end{equation}
	This ensures that $\Omega_\lambda$ satisfies the condition of Theorem \ref{th:Main-Theorem} with $ \tau = 1 / 4 $ and the prescribed $ \gamma $.
	So, the claim follows from Theorem \ref{th:Main-Theorem}.
	
\end{proof}

In particular,
\begin{equation}
	\left( \inrad(\Omega_\lambda) \sqrt{\lambda} \right)^{\frac{2n}{n-2}} \geq C \| \eigenf \|_{L^2(\Omega_\lambda)}
\end{equation}
and summing over all nodal domains yields
\begin{equation}
	\sum_{\Omega_\lambda} \inrad(\Omega_\lambda)^{\frac{2n}{n-2}} \geq \frac{C}{\lambda^{\frac{n}{2n-4}}},
\end{equation}
with the constant $C$ being better than the constant $C_1$, appearing in Theorem \ref{th:Main-Theorem}. This allows one to obtain an estimate on the generalized mean with exponent $ \frac{n}{n-2} $ of all the inner radii corresponding to different nodal domains.

Note that the main obstruction against the application of Theorem \ref{th:Main-Theorem} is the fact that one needs to know that the $L^2$-norm of $\eigenf$ over $\Omega_\lambda$ is mainly contained in good cubes and this should be uniform w.r.t. $\lambda $ (or at least conveniently controlled).

As further questions one might ask whether a dissipation of the bad cubes is to be expected in some special cases (e.g. the case of ergodic geodesic flow) - that is, is it true that a nodal domain should have a well-distributed $ L^2 $ norm in the sense of Theorem $ \ref{th:Main-Theorem} $?

A relaxed version of this question is, of course, a probabilistic statement of the kind - a significant amount of nodal domains should enjoy the property of having well-distributed $ L^2 $ norm.

\end{document}